\documentclass[11pt,letterpaper]{amsart}

\makeatletter
\def\subsection{\@startsection{subsection}{2}%
  \z@{.7\linespacing\@plus.7\linespacing}{.2\linespacing}%
  {\centering\normalfont\scshape}}
\makeatother

\usepackage[letterpaper,top=3.7cm, bottom=3.7cm, left=3.7cm, right=3.7cm]{geometry}

\usepackage{amssymb}
\usepackage{mathtools}
\usepackage{mathabx}
\usepackage{enumerate}
\usepackage{bbm}
\usepackage{ytableau}

\newtheorem{theorem}{Theorem}[section]
\newtheorem{proposition}[theorem]{Proposition}
\newtheorem{corollary}[theorem]{Corollary}
\newtheorem{lemma}[theorem]{Lemma}

\theoremstyle{definition}
\newtheorem{definition}[theorem]{Definition}
\newtheorem{example}[theorem]{Example}
\newtheorem{remark}[theorem]{Remark}
\newtheorem*{remark*}{Remark}

\renewcommand{\AA}{\mathbb{A}}

\newcommand{\FF}{\mathbb{F}}
\newcommand{\CC}{\mathbb{C}}
\newcommand{\RR}{\mathbb{R}}

\newcommand{\Fc}{\mathcal{F}}

\newcommand{\Ic}{\mathcal{I}}
\newcommand{\Jc}{\mathcal{J}}

\newcommand{\1}{\mathbbm{1}}
\newcommand{\e}{\epsilon}

\newcommand{\udelta}{\underline{\delta}}
\newcommand{\ukappa}{\underline{\kappa}}
\newcommand{\ulambda}{\underline{\lambda}}
\newcommand{\umu}{\underline{\mu}}
\newcommand{\unu}{\underline{\nu}}

\newcommand{\sums}[1]{\sum_{\substack{#1}}}

\newcommand{\ang}[1]{\langle#1\rangle}
\newcommand{\norm}[1]{\lVert#1\rVert}

\newcommand{\abs}[1]{\lvert#1\rvert}
\newcommand{\bigabs}[1]{\big\lvert#1\big\rvert}

\DeclareMathOperator{\Ind}{Ind}

\DeclareMathOperator{\rk}{rk}

\DeclareMathOperator{\GL}{GL}
\DeclareMathOperator{\GammaL}{\Gamma L}
\DeclareMathOperator{\SL}{SL}

\DeclareMathOperator{\Par}{Par}

\DeclareMathOperator{\type}{type}

\DeclareMathOperator{\Sym}{Sym}

\renewcommand{\Re}{\operatorname{Re}}
\renewcommand{\Im}{\operatorname{Im}}

\begin{document}

\title{Transitivity in finite general linear groups}

\author{Alena Ernst and Kai-Uwe Schmidt}
\address{Department of Mathematics, Paderborn University, Warburger Str.\ 100, 33098 Paderborn, Germany.}
\email{alena.ernst@math.upb.de}
\email{kus@math.upb.de}

\thanks{Funded by the Deutsche Forschungsgemeinschaft (DFG, German Research Foundation) -- Project number 459964179.}

\date{16 September 2022}


\begin{abstract}

It is known that the notion of a transitive subgroup of a permutation group $G$ extends naturally to subsets of $G$. We consider subsets of the general linear group $\GL(n,q)$ acting transitively on flag-like structures, which are common generalisations of $t$-dimensional subspaces of $\FF_q^n$ and bases of $t$-dimensional subspaces of~$\FF_q^n$. We give structural characterisations of transitive subsets of $\GL(n,q)$ using the character theory of $\GL(n,q)$ and interprete such subsets as designs in the conjugacy class association scheme of $\GL(n,q)$. In particular we generalise a theorem of Perin on subgroups of $\GL(n,q)$ acting transitively on $t$-dimensional subspaces. We survey transitive subgroups of $\GL(n,q)$, showing that there is no subgroup of $\GL(n,q)$ with $1<t<n$ acting transitively on $t$-dimensional subspaces unless it contains $\SL(n,q)$ or is one of two exceptional groups. On the other hand, for all fixed $t$, we show that there exist nontrivial subsets of $\GL(n,q)$ that are transitive on linearly independent $t$-tuples of $\FF_q^n$, which also shows the existence of nontrivial subsets of $\GL(n,q)$ that are transitive on more general flag-like structures. We establish connections with orthogonal polynomials, namely the Al-Salam-Carlitz polynomials, and generalise a result by Rudvalis and Shinoda on the distribution of the number of fixed points of the elements in~$\GL(n,q)$. Many of our results can be interpreted as $q$-analogs of corresponding results for the symmetric group.
\end{abstract}

\maketitle

\thispagestyle{empty}


\section{Introduction}

A subgroup $G$ of the symmetric group $\Sym(n)$ is \emph{$t$-homogeneous} if $G$ is transitive on the subsets of $\{1,2,\dots,n\}$ with $t$ elements. Livingstone and Wagner~\cite{LivWag1965} proved the following result.
\begin{theorem}[{\cite{LivWag1965}}]
\label{thm:LivWag}
Let $G$ be a subgroup of $\Sym(n)$ that is $t$-homogeneous for some~$t$ satisfying $1\le t\le n/2$. Then $G$ is also $(t-1)$-homogeneous.
\end{theorem}
\par
This theorem was generalised by Martin and Sagan~\cite{MarSag2007} in various ways. Their first generalisation replaces subgroups of $\Sym(n)$ by subsets of $\Sym(n)$. Let $G$ be a group acting on a set $\Omega$. We say that a subset $Y$ of $G$ is \emph{transitive} on $\Omega$ if there is a constant~$r$ such that the following holds. For all $a,b\in\Omega$, there are exactly $r$ elements $g\in Y$ such that $ga=b$. If $Y$ is a subgroup of $G$, then this notion coincides with that of a transitive group action of $Y$ on $\Omega$. The second generalisation of~\cite{MarSag2007} replaces subsets of $\Omega$ by set partitions of $\Omega$. An (integer) \emph{partition} of a natural number $n$ is a sequence $\lambda=(\lambda_1,\lambda_2,\dots)$ of nonnegative integers that sum up to~$n$ and satisfy $\lambda_1\ge\lambda_2\ge\cdots$; these numbers are called the \emph{parts} of $\lambda$. For two partitions $\lambda=(\lambda_1,\lambda_2,\dots)$ and $\mu=(\mu_1,\mu_2,\dots)$ of $n$, we say that $\lambda$ \emph{dominates} $\mu$ and write $\mu\unlhd\lambda$ if
\begin{equation}
\sum_{i=1}^k\mu_i\le \sum_{i=1}^k\lambda_i\quad\text{for each $k\ge 1$}.   \label{eqn:dominance_order}
\end{equation}
Finally, for a  partition $\lambda=(\lambda_1,\lambda_2,\dots)$ of $n$, a \emph{$\lambda$-partition} is an ordered partition of the set $\{1,2,\dots,n\}$ into subsets of cardinality $\lambda_1,\lambda_2,\dots$. The following is one of the main results in~\cite{MarSag2007}.
\begin{theorem}[{\cite{MarSag2007}}]
Let $Y$ be a subset of $\Sym(n)$ that is transitive on $\sigma$-partitions. Then~$G$ is also transitive on $\tau$-partitions for all $\tau$ satisfying $\sigma\unlhd\tau$.
\end{theorem}
\par
In fact,~\cite{MarSag2007} systematically studies subsets of $\Sym(n)$ that are transitive on $\sigma$-partitions using the character theory of $\Sym(n)$. In many aspects, this paper studies $q$-analogous problems, namely we replace $\Sym(n)$ by the general linear group $\GL(n,q)$, consisting of the invertible $n\times n$ matrices with entries in $\FF_q$. Our starting point is the following $q$-analog of Theorem~\ref{thm:LivWag}, proved by Perin~\cite{Per1972} (who attributed it to an unpublished result by McLaughlin). Henceforth a \emph{$t$-space} is a $t$-dimensional subspace of $\FF_q^n$. 
\begin{theorem}[\cite{Per1972}]
Let $G$ be a subgroup of $\GL(n,q)$ that is transitive on $t$-spaces for some~$t$ satisfying $1\le t\le n/2$. Then $G$ is also transitive on $(t-1)$-spaces.
\end{theorem}
\par
An (integer) \emph{composition} of a nonnegative number $n$ is much like a partition of $n$, except that the sequence entries are not necessarily nonincreasing. For a composition $\lambda=(\lambda_1,\lambda_2,\dots)$ of $n$, a \emph{$\lambda$-flag} is a sequence of subspaces $(V_1,V_2,\dots)$ of $\FF_q^n$ such that
\[
\{0\}=V_0\le V_1\le V_2\le\cdots
\]
and $\dim(V_i/V_{i-1})=\lambda_i$ for each $i\ge 1$. The following is an example of the results we obtain.
\begin{theorem}
\label{thm:main}
Let $Y$ be a subset of $\GL(n,q)$ that is transitive on $\sigma$-flags. Then $Y$ is also transitive on $\tau$-flags for all compositions $\tau$ satisfying $\tilde\sigma\unlhd\tilde\tau$, where $\tilde\sigma$ and $\tilde\tau$ are the respective partitions obtained from $\sigma$ and $\tau$ by rearranging the parts.
\end{theorem}
\par
In fact we consider subsets of $\GL(n,q)$ that are transitive on more general objects, namely on pairs $(F,B)$, where $F$ is a $\sigma$-flag and $B$ is a tuple whose entries are ordered bases of some of the nonzero quotient spaces $V_i/V_{i-1}$. Such objects are called \emph{$(\sigma,\Ic)$-flags}, where $\Ic$ indexes the quotient spaces whose bases occur in $B$. We study $(\sigma,\Ic)$-flags using the character theory of $\GL(n,q)$ and the theory of association schemes (see~\cite{BanIto1984}, for example).
\par
We organise this paper in the following way. In Section~\ref{sec:GL} we recall some relevant background on the character theory of $\GL(n,q)$ and the conjugacy class association scheme of $\GL(n,q)$. In Section~\ref{sec:Designs} we give structural results for transitive subsets of $\GL(n,q)$. In particular we show that such subsets can be characterised as designs in the corresponding association scheme, leading directly to results like Theorem~\ref{thm:main}. In Section~\ref{sec:subgroups} we survey the examples coming from subgroups. In Section~\ref{sec:Cliques} we study so-called \emph{cliques} in $\GL(n,q)$, which are subsets of $\GL(n,q)$ such that, for all distinct elements $x,y$ in the subset, $x^{-1}y$ fixes no $(\sigma,\Ic)$-flag. Among other things, this will allow us to establish the nonexistence of sharply transitive subsets of $\GL(n,q)$ in certain cases. In Section~\ref{sec:Existence} we show the existence of small subsets of $\GL(n,q)$ that are transitive on $(\sigma,\Ic)$-flags, as long as the largest part of $\sigma$ is large compared to the sum of all other parts of $\sigma$. We find this interesting since subgroups of $\GL(n,q)$ that are transitive on $(\sigma,\Ic)$-flags are extremely rare. In Section~\ref{sec:designs_codes_polynomials} we discuss connections between transitive subsets of $\GL(n,q)$ and cliques in $\GL(n,q)$ on one hand and certain orthogonal polynomials, namely the Al-Salam-Carlitz polynomials, on the other hand.


\section{The finite general linear groups}
\label{sec:GL}

We shall give a brief account of the conjugacy classes and the (complex) irreducible characters of the general linear group $\GL(n,q)$. We mostly follow~\cite[Ch.~IV]{Mac1979}.

\subsection{Partitions}

An (integer) \emph{partition} is a sequence $\lambda=(\lambda_1,\lambda_2,\dots)$ of nonnegative integers that sum up to a finite number and satisfy $\lambda_1\ge\lambda_2\ge\cdots$. The \emph{size} of $(\lambda_1,\lambda_2,\dots)$ is defined to be $\abs{\lambda}=\lambda_1+\lambda_2+\cdots$ and its \emph{length} $\ell(\lambda)$ is the largest $i$ such that $\lambda_i>0$. We often write $(\lambda_1,\lambda_2,\dots,\lambda_{\ell(\lambda)})$ instead of $(\lambda_1,\lambda_2,\dots)$. If $\abs{\lambda}=n$, then we also say that $\lambda$ is a partition of $n$. Let $\Par$ be the set of integer partitions. We denote the unique partition of $0$ by $\varnothing$.
\par
The \emph{Young diagram} of a partition $(\lambda_1,\lambda_2,\dots,\lambda_k)$ of $n$ is an array of~$n$ boxes with left-justified rows and top-justified columns, where row $i$ contains $\lambda_i$ boxes. To each partition $\lambda$ belongs a \emph{conjugate partition} $\lambda'$ whose parts are the number of boxes in the columns of the Young diagram of $\lambda$.
\par
We recall three partial orders on integer partitions, namely containment, refinement, and dominance order. Let $\lambda,\mu\in\Par$ be two partitions. We say that $\lambda$ \emph{contains}~$\mu$ and write $\mu\subseteq\lambda$ if the Young diagram of $\mu$ is contained in the Young diagram of~$\lambda$. We say that $\mu$ \emph{refines} $\lambda$ if $\abs{\mu}\le\abs{\lambda}$ and the parts of $\lambda$ can be partitioned to produce the parts of $(\mu_1,\dots,\mu_{\ell(\mu)},1^{\abs{\lambda}-\abs{\mu}})$. For example $(321)$ refines $(7422)$. We say that $\lambda$ \emph{dominates} $\mu$ and write $\mu\unlhd\lambda$ if~\eqref{eqn:dominance_order} holds. As usual we write $\mu\lhd\lambda$ if $\mu\unlhd\lambda$ and $\mu\ne\lambda$. Typically these partial orders are only defined for partitions of the same size, but it is natural to extend these to the set of all partitions. 

\subsection{Conjugacy classes}
\label{sec:conjugacy_classes}

We shall now describe the conjugacy classes of $\GL(n,q)$ (see~\cite[Ch.~IV,\S~3]{Mac1979}, for example). Let $\Phi$ be the set of monic irreducible polynomials in $\FF_q[X]$ distinct from~$X$. We shall often write $1$ instead of $X-1$ when the meaning is clear from the context. We also write $\abs{f}$ for the degree of $f\in\Phi$.
\par
Let $\Lambda$ be the set of mappings $\ulambda\colon\Phi\to\Par$ of finite support (with $\varnothing$ being the zero element in $\Par$). We often use the short-hand notation $f\mapsto\lambda$ for the element $\ulambda\in\Lambda$ that is supported only on $f$ and satisfies $\ulambda(f)=\lambda$. We define the \emph{size} of an element $\ulambda\in\Lambda$ to be
\[
\norm{\ulambda}=\sum_{f\in\Phi}\abs{\ulambda(f)}\cdot\abs{f}
\]
and put $\Lambda_n=\{\ulambda\in\Lambda:\norm{\ulambda}=n\}$.
\par
The \emph{companion matrix} of $f\in\Phi$ with $f=X^d+f_{d-1}X^{d-1}+\cdots+f_1X+f_0$ is
\[
C(f)=\begin{bmatrix}
   &    &  &    & -f_0\\
 1 &    &  &    & -f_1\\
   & 1  &  &    & -f_2\\
   &    & \ddots & & \vdots\\
   &    &  & 1 & -f_{d-1}\\
\end{bmatrix}\in\FF_q^{d\times d}.
\]
(where blanks are filled with zeros). For $f\in\Phi$ of degree $d$ and a positive integer $k$, we write
\[
C(f,k)=\begin{bmatrix}
C(f) & I &   & & \\
  & C(f) & I & & \\
&&\ddots&\ddots\\
  &   &   & \ddots &  I\\
  &   &   & &  C(f) 
\end{bmatrix}\in\FF_q^{kd\times kd},
\]
and for $f\in\Phi$ and $\mu\in\Par$, we define $C(f,\mu)$ to be the block diagonal matrix of order $\abs{\mu}\cdot\abs{f}$ with blocks $C(f,\mu_1),C(f,\mu_2),\dots$. Finally, with every $\umu\in\Lambda_n$ we associate the block diagonal matrix $R_{\umu}$ of order $n$ whose blocks are $C(f,\umu(f))$, where~$f$ ranges through the support of $\umu$. Then every element $g$ of $\GL(n,q)$ is conjugate to exactly one matrix $R_{\umu}$ for $\umu\in\Lambda_n$, which is called the \emph{Jordan canonical form} of $g$. Hence~$\Lambda_n$ indexes the conjugacy classes of $\GL(n,q)$; we denote by $C_{\umu}$ the conjugacy class containing $R_{\umu}$. Note that $C_{X-1\mapsto (1^n)}$ is the conjugacy class containing the identity.

\subsection{Parabolic induction}

A \emph{composition} is much like a partition, except that the parts do not need to be nonincreasing. Let $\lambda=(\lambda_1,\lambda_2,\dots,\lambda_k)$ be a composition of $n$. Let $P_\lambda$ be the parabolic subgroup of $\GL(n,q)$ consisting of block upper-triangular matrices with block sizes $\lambda_1,\lambda_2,\dots,\lambda_k$, namely
\[
P_\lambda=\left\{\begin{bmatrix}
A_1 & *     & \cdots & *\\
       & A_2 & \cdots & *\\
       &        & \ddots & \vdots\\
       &        &            & A_k\\
\end{bmatrix}
:A_i\in\GL(\lambda_i,q)
\right\}.
\]
Let $\pi_i:P_\lambda\to\GL(\lambda_i,q)$ be the mapping that projects to the $i$-th block on the diagonal, so that
\begin{equation}
\pi_i:
\begin{bmatrix}
A_1 & *     & \cdots & *\\
       & A_2 & \cdots & *\\
       &        & \ddots & \vdots\\
       &        &            & A_k\\
\end{bmatrix}
\mapsto A_i.   \label{eqn:projections}
\end{equation}
Let $\phi_i$ be a class function of $\GL(\lambda_i,q)$. Then 
\[
\prod_{i=1}^k(\phi_i\circ \pi_i)
\]
is a class function of $P_\lambda$. We define the product $\phi_1\cdot\phi_2\cdots\phi_k$ to be the induction of this class function to $\GL(n,q)$, that is
\[
\phi_1\cdot\phi_2\cdots\phi_k=\Ind_{P_\lambda}^{\GL(n,q)}\left(\prod_{i=1}^k(\phi_i\circ \pi_i)\right).
\]

\subsection{The irreducible characters}

The complete set of (complex) irreducible characters has been obtained by Green~\cite{Gre1956}. Good treatments of this topic are also contained in~\cite[Ch.~IV]{Mac1979} and~\cite{Jam1986}. The irreducible characters of $\GL(n,q)$ are naturally indexed by $\Lambda_n$ and, for $\ulambda\in\Lambda_n$, we denote by $\chi^{\ulambda}$ the corresponding irreducible character. For $f\in\Phi$ and a partition $\lambda$, the characters $\chi^{f\mapsto\lambda}$ are typically called the \emph{primary} irreducible characters of $\GL(n,q)$. It~is well known (see~\cite[\S~8]{Jam1986}, for example) that the irreducible characters of $\GL(n,q)$ satisfy
\[
\chi^{\ulambda}=\prod_{f\in\Phi}\chi^{f\mapsto\ulambda(f)}.
\]
We use the indexing of~\cite{Jam1986}, so that $\chi^{f\mapsto(n)}$ is a cuspidal character of $\GL(n,q)$ and in particular $\chi^{X-1\mapsto(n)}$ is the trivial character. In contrast, $\lambda$ is replaced by the conjugate partition $\lambda'$ in~\cite[Ch.~IV]{Mac1979}. We often denote the degree of $\chi^{\ulambda}$ by $f_{\ulambda}$.
\par
It follows from~\cite[Ch.~IV]{Mac1979} that, for each $f\in\Phi$, the algebra with multiplication~$\cdot$~generated by $\{\chi^{f\mapsto \lambda}:\lambda\in\Par\}$ is isomorphic to the algebra of symmetric functions with~$\chi^{f\mapsto \lambda}$ being sent to the Schur function~$s_\lambda$. In particular the decomposition of the product $\chi^{f\mapsto \lambda}\cdot \chi^{f\mapsto \nu}$ into irreducible characters is given by the Littlewood-Richardson rule~\cite[Ch.~I]{Mac1979}. Let $\lambda,\mu,\nu\in\Par$ and let $c_{\lambda\nu}^\mu$ be the number of semistandard skew-tableaux~$T$ of shape $\mu/\lambda$ and content~$\nu$ such that the sequence, obtained by concatenating its reversed rows, is a lattice permutation. Note that we have $c_{\lambda\nu}^\mu=0$ unless $\abs{\mu}=\abs{\lambda}+\abs{\nu}$ and $\lambda,\nu\subseteq \mu$. The Littlewood-Richardson rule then states that
\[
\chi^{f\mapsto \lambda}\cdot \chi^{f\mapsto \nu}=\sum_{\mu\in\Par}c_{\lambda\nu}^\mu\,\chi^{f\mapsto \mu}
\]
for all $f\in\Phi$. Now, for $\ulambda,\umu,\unu\in\Lambda$, define
\[
c_{\ulambda\unu}^{\umu}=\prod_{f\in\Phi}c_{\ulambda(f)\unu(f)}^{\umu(f)}
\]
and note that $c_{\ulambda\unu}^{\umu}=0$ unless $\norm{\umu}=\norm{\ulambda}+\norm{\unu}$ and $\ulambda,\unu\subseteq \umu$, where $\ulambda\subseteq \umu$ means $\ulambda(f)\subseteq\umu(f)$ for all $f\in\Phi$. The following lemma is then immediate.
\begin{lemma}
\label{lem:LLR_general}
For all $\ulambda,\unu\in\Lambda$ we have
\[
\chi^{\ulambda}\cdot\chi^{\unu}=\sum_{\umu\in\Lambda}c_{\ulambda\unu}^{\umu}\,\chi^{\umu}.
\]
\end{lemma}
\par
We shall also need the following straightforward result.
\begin{remark}
\label{rem:LLR}
Let $\ulambda,\umu\in\Lambda$ such that $\ulambda(f)\subseteq \umu(f)$ for all $f\in\Phi$. Then there exists $\unu\in\Lambda$ such that $c_{\ulambda\unu}^{\umu}>0$.
\end{remark}
\par
The primary irreducible characters must also obey Young's rule~\cite[Ch.~I]{Mac1979}. For partitions $\lambda$ and $\mu$ of the same size, the \emph{Kostka number} $K_{\lambda\mu}$ is the number of semistandard Young tableaux of shape $\lambda$ and content $\mu$. 
\begin{lemma}
\label{lem:Youngs_rule}
For each $f\in\Phi$ and each partition $\mu=(\mu_1,\mu_2,\dots)$, we have
\[
\prod_{i\ge 1}\chi^{f\mapsto (\mu_i)}=\sum_{\lambda\unrhd\mu} K_{\lambda\mu}\,\chi^{f\mapsto \lambda},
\]
where $\lambda$ ranges over the partitions of $\abs{\mu}$.
\end{lemma}


\subsection{The conjugacy class association scheme}
\label{sec:association_schemes}

We shall study the combinatorics in $\GL(n,q)$ from the viewpoint of an association scheme. We refer to~\cite{BanIto1984} for background on association schemes. Every finite group gives rise to an association scheme (see~\cite[Section 2.7]{BanIto1984} for details), called the \emph{conjugacy class association scheme} of the group, but the theory of association schemes is much more general than that. We shall recall relevant background about the conjugacy class association scheme of~$\GL(n,q)$. 
\par
Henceforth we use the following notation. For a field $K$ and finite sets $X$ and~$Y$, we denote by $K(X,Y)$ the set of $\abs{X}\times\abs{Y}$ matrices $A$ with entries in $K$, where rows and columns are indexed by $X$ and $Y$, respectively. For $x\in X$ and $y\in Y$, the $(x,y)$-entry of $A$ is written as $A(x,y)$. If $\abs{Y}=1$, then we omit $Y$, so $K(X)$ is the set of column vectors $a$ indexed by $X$ and, for $x\in X$, the $x$-entry of $a$ is written as $a(x)$.
\par
For $\umu\in\Lambda_n$, let $D_{\umu}\in\CC(\GL(n,q),\GL(n,q))$ be given by
\begin{equation}
D_{\umu}(x,y)=\begin{cases}
1 & \text{for $x^{-1}y\in C_{\umu}$}\\
0 & \text{otherwise}.
\end{cases}   \label{eqn:D_matrices}
\end{equation}
Let $\AA$ be the vector space generated by $\{D_{\umu}:\umu\in\Lambda_n\}$ over the complex numbers. Then $\AA$ is a commutative matrix algebra, which contains the identity and is closed under conjugate transposition. The collection of zero-one matrices~$D_{\umu}$ therefore defines an association scheme, called the \emph{conjugacy class association scheme} of $\GL(n,q)$. The algebra $\AA$ is called the \emph{Bose-Mesner algebra} of this association scheme.
\par
Since $\AA$ is commutative, it can be simultaneously diagonalised and therefore there exists a basis $\{E_{\ulambda}:\ulambda\in\Lambda_n\}$ of $\AA$ consisting of Hermitian matrices with the minimal idempotent property
\begin{equation}
E_{\ukappa}E_{\ulambda}=\delta_{\ukappa\ulambda}E_{\ulambda}.   \label{eqn:E_orthogonal}
\end{equation}
These matrices are given by~\cite[Theorem II.7.2]{BanIto1984}
\[
E_{\ulambda}\,=\frac{f_{\ulambda}}{\abs{\GL(n,q)}}\sum_{\umu\in\Lambda_n}\chi^{\ulambda}_{\umu}\,D_{\umu},
\]
where $\chi^{\ulambda}_{\umu}$ is the irreducible character $\chi^{\ulambda}$ evaluated on the conjugacy class $C_{\umu}$ and $f_{\ulambda}$ is the degree of this character. By~\eqref{eqn:D_matrices} the entries of $E_{\ulambda}$ are given by
\begin{equation}
E_{\ulambda}(x,y)=\frac{f_{\ulambda}}{\abs{ \GL(n,q)}}\chi^{\ulambda}(x^{-1}y).   \label{eqn:E_matrix_chi}
\end{equation}
Note that $E_{X-1\mapsto (n)}$ is the all-ones matrix and that $\sum_{\ulambda\in\Lambda_n}E_{\ulambda}$ is the identity matrix, which can be seen using standard properties of characters. Since the matrices $E_{\ulambda}$ are idempotent, their eigenvalues are $0$ or $1$, and so their ranks sum up to $\abs{\GL(n,q)}$. Let $V_{\ulambda}$ be the column space of $E_{\ulambda}$. It follows from~\eqref{eqn:E_orthogonal} that these vector spaces are pairwise orthogonal and that
\begin{equation}
\CC(\GL(n,q))=\bigoplus_{\ulambda\in\Lambda_n}V_{\ulambda}.   \label{eqn:CC_Gn_decomp}
\end{equation}
\par
Now let $Y$ be a subset of $\GL(n,q)$. The \emph{inner distribution} of $Y$ is the tuple $(a_{\umu})_{\umu\in\Lambda_n}$, where
\begin{equation}
a_{\umu}=\frac{1}{\abs{Y}}\,\sum_{x,y\in Y}D_{\umu}(x,y),   \label{eqn:inner_dist}
\end{equation}
and the \emph{dual distribution} of $Y$ is the tuple $(a'_{\ulambda})_{\ulambda\in\Lambda_n}$, where
\begin{equation}
\label{eqn:dual_dist}
a'_{\ulambda}=\frac{\abs{\GL(n,q)}}{\abs{Y}}\,\sum_{x,y\in Y}E_{\ulambda}(x,y).
\end{equation}
Explicitly using~\eqref{eqn:E_matrix_chi} we have
\begin{equation}
a'_{\ulambda}=\frac{f_{\ulambda}}{\abs{Y}}\sum_{x,y\in Y}\chi^{\ulambda}(x^{-1}y) . \label{eqn:dual_dist_explicit}
\end{equation}
The entries in the inner distribution are clearly nonnegative numbers. The same holds for the entries in the dual distribution. To see this, let $\1_Y\in\CC(\GL(n,q))$ be the characteristic vector of $Y$, so that $\1_Y(x)=1$ if $x\in Y$ and $\1_Y(x)=0$ otherwise. Since~$E_{\ulambda}$ is Hermitian and idempotent, we find from~\eqref{eqn:dual_dist} that
\[
\frac{\abs{Y}}{\abs{\GL(n,q)}}\,a'_{\ulambda}=\1^\top_YE_{\ulambda}\1_Y=\1^*_YE^*_{\ulambda}E_{\ulambda}\1_Y=\norm{E_{\ulambda}\1_Y}^2.
\]
This shows that the entries in the dual distribution are real and nonnegative. Moreover the extreme case $a'_{\ulambda}=0$ occurs if and only if $\1_Y$ is orthogonal to $V_{\ulambda}$.
\par
The power of association schemes in combinatorics stems from the observation that interesting combinatorial structures can often be characterised as subsets of association schemes for which certain entries in the inner or dual distribution are equal to zero. Delsarte~\cite{Del1973} calls these objects \emph{cliques} and \emph{designs}, respectively. We shall see that interesting subsets of $\GL(n,q)$ indeed are cliques or designs in the conjugacy class association scheme of $\GL(n,q)$.


\section{Notions of transitivity}
\label{sec:Designs}

We consider pairs $(\rho,\Ic)$, where $\rho$ is a composition of $n$ and~$\Ic$ is a subset of $\{1,2,\dots,\ell(\rho)\}$ and, if $q=2$, we insist that $\rho_i>1$ for each $i\not\in\Ic$. We denote the collection of such pairs by $\Sigma_{n,q}$. 
\par
For a composition $\rho$ of $n$, a \emph{$\rho$-flag} is a tuple of subspaces $(V_1,V_2,\dots,V_{\ell(
\rho)})$ of $\FF_q^n$ such that
\[
\{0\}=V_0\le V_1\le V_2\le\cdots\le V_{\ell(\rho)}=\FF_q^n
\]
and $\dim(V_i/V_{i-1})=\rho_i$ for each $i\in\{1,2,\dots,\ell(\rho)\}$. Let $\alpha=(\rho,\Ic)$ be an element of $\Sigma_{n,q}$ with $\Ic=\{i_1,i_2,\dots,i_k\}$. We define an \emph{$\alpha$-flag} to be a pair $(F,B)$, where $F=(V_1,V_2,\dots,V_{\ell(\rho)})$ is a $\rho$-flag and $B=(B_1,B_2,\dots,B_k)$ is a tuple of ordered bases of $V_{i_1}/V_{{i_1}-1},V_{i_2}/V_{{i_2}-1},\dots,V_{i_k}/V_{{i_k}-1}$ with $V_0=\{0\}$. For example, a $((t,n-t),\varnothing)$-flag is essentially a $t$-dimensional subspace of $\FF_q^n$ and a $((t,n-t),\{1\})$-flag is essentially a $t$-tuple of linearly independent elements of $\FF_q^n$.
\par
Let $\Omega$ be a set on which $\GL(n,q)$ acts. We say that a subset $Y$ of $\GL(n,q)$ is \emph{transitive} on $\Omega$ if there is a constant $r$ such that the following holds. For all $a,b\in \Omega$, there are exactly $r$ elements $g\in Y$ such that $ga=b$. If $r=1$, then we also call $Y$ \emph{sharply} transitive on $\Omega$. If $Y$ is a subgroup of $\GL(n,q)$, then this notion coincides with that of a transitive group action of $Y$ on~$\Omega$.
\par
We are interested in subsets of $\GL(n,q)$ that are transitive on $\alpha$-flags for $\alpha\in\Sigma_{n,q}$. The following example gives a simple construction of subsets of $\GL(n,q)$ that are sharply transitive on $((1,n-1),\Ic)$-flags for $\Ic=\{1\}$ and $\Ic=\varnothing$.
\begin{example}
\label{ex:Singer}
Let $C\in\GL(n,q)$ be the companion matrix of an irreducible polynomial in $\FF_q[X]$ of degree $n$. Then $\FF_q[C]$ is a representation of $\FF_{q^n}$ over $\FF_q$ and so the multiplicative group $\FF_q[C]^*$ of $\FF_q[C]$ is sharply transitive on $\FF_q^n\setminus\{0\}$. Hence $\FF_q[C]^*$ is sharply transitive on $((1,n-1),\{1\})$-flags. Of course $\FF_q[C]^*$ is a cyclic of subgroup of $\GL(n,q)$, known as a \emph{Singer cycle}. Moreover $\FF_q[C]^*$ contains a cyclic subgroup of order $(q^n-1)/(q-1)$ that is sharply transitive on the one-dimensional subspaces of~$\FF_q^n$. Hence this subgroup is sharply transitive on $((1,n-1),\varnothing)$-flags.
\par
More generally, a subset of $\GL(n,q)$ that is sharply transitive on $\FF_q^n\setminus\{0\}$ is equivalent to each of the following objects: a spread set in $\FF_q^{n\times n}$, a finite quasifield of order~$q^n$, and a finite translation plane of order $q^n$ (see~\cite[\S~5.1]{Dem1968}, for example).
\end{example}
\par
In what follows we give a structural interpretation of subsets of $\GL(n,q)$ that are transitive on $\alpha$-flags (for $\alpha\in\Sigma_{n,q}$). To do so, we require a few definitions.
\par
With each $(\rho,\Ic)\in\Sigma_{n,q}$ we associate a pair of partitions $(\sigma,\tau)$, called the \emph{type} of~$(\rho,\Ic)$, where $\sigma$ is the partition whose parts are those $\rho_i$ with $i\in\Ic$ and $\tau$ is the partition whose parts are those $\rho_i$ with $i\not\in \Ic$. For example $((25123),\{2,3,5\})$ has type $((531),(22))$. We denote the type of $\alpha\in\Sigma_{n,q}$ by $\type(\alpha)$ and the set of possible such types by $\Theta_{n,q}$. Hence $\Theta_{n,q}$ contains all pairs of partitions $(\sigma,\tau)$ such that $\abs{\sigma}+\abs{\tau}=n$ and all parts of $\tau$ are strictly larger than $1$ for $q=2$.  
\par
We also define the \emph{type} of $\ulambda\in\Lambda_n$ as a pair of partitions $(\kappa,\lambda)$, where $\lambda=\ulambda(1)$ and~$\kappa$ has $\abs{\ulambda(f)}$ parts of size $\abs{f}$ as $f$ ranges through $\Phi\setminus\{X-1\}$. For example, when $q=3$, the type of $\ulambda\in\Lambda_n$ given by
\[
X-1\mapsto (31),\quad X+1\mapsto (31),\quad X^2+1\mapsto (2),\quad X^2+X-1\mapsto (21)
\]
equals $((2^51^4),(31))$. We denote the type of $\ulambda\in\Lambda_n$ by $\type(\ulambda)$. Note that, if $(\kappa,\lambda)$ is the type of $\ulambda\in\Lambda_n$, then $\abs{\kappa}+\abs{\lambda}=n$. Note that the unique irreducible character~$\chi^{\ulambda}$ of $\GL(n,q)$ with $\type(\ulambda)=(\varnothing,(n))$ is the trivial character.
\par
We define a partial order on pairs of partitions by 
\[
(\nu,\mu)\preceq (\kappa,\lambda)\quad\Leftrightarrow\quad \text{$\kappa$ refines $\nu$ and $\mu\unlhd\lambda$}
\]
and write $(\nu,\mu)\prec (\kappa,\lambda)$ if $(\nu,\mu)\preceq (\kappa,\lambda)$ and $(\nu,\mu)\ne(\kappa,\lambda)$. The following result gives a characterisation of subsets of $\GL(n,q)$ that are transitive on $\alpha$-flags as a design in the corresponding conjugacy class association scheme in the sense of Delsarte~\cite{Del1973}.
\begin{theorem}
\label{thm:designs_dual_dist}
Let $Y$ be a subset of $\GL(n,q)$ with dual distribution $(a'_{\ulambda})$ and let $\alpha\in\Sigma_{n,q}$. Then $Y$ is transitive on $\alpha$-flags if and only if
\[
a'_{\ulambda}=0\quad\text{for all $\ulambda\in\Lambda_n$ satisfying $\type(\alpha)\preceq\type(\ulambda)\prec (\varnothing,(n))$}.
\]
\end{theorem}
\par
Before we prove Theorem~\ref{thm:designs_dual_dist}, we discuss some of its consequences. The first one is that transitivity on $\alpha$-flags depends only on the type of $\alpha$.
\begin{corollary}
\label{cor:design_transitive_flags}
Let $\alpha,\beta\in\Sigma_{n,q}$ be of the same type and let $Y$ be a subset of $\GL(n,q)$. Then $Y$ is transitive on $\alpha$-flags if and only if $Y$ is transitive on $\beta$-flags.
\end{corollary}
\par
Corollary~\ref{cor:design_transitive_flags} motivates the following definition.
\begin{definition}
For $(\sigma,\tau)\in\Theta_{n,q}$, we define a subset $Y$ of $\GL(n,q)$ to be \emph{$(\sigma,\tau)$-transitive} if $Y$ is transitive on the set of $\alpha$-flags for some $\alpha\in\Sigma_{n,q}$ of type $(\sigma,\tau)$.
\end{definition}
\par
Note that Example~\ref{ex:Singer} gives $(\sigma,\tau)$-transitive sets for $(\sigma,\tau)$ equal to $((1),(n-1))$ and $(\varnothing,(n-1,1))$. We may now restate Theorem~\ref{thm:designs_dual_dist} as follows.
\begin{corollary}
\label{cor:designs_dual_dist}
Let $Y$ be a subset of $\GL(n,q)$ with dual distribution $(a'_{\ulambda})$ and let $(\sigma,\tau)\in\Theta_{n,q}$. Then $Y$ is $(\sigma,\tau)$-transitive if and only if
\[
a'_{\ulambda}=0\quad\text{for all $\ulambda\in\Lambda_n$ satisfying $(\sigma,\tau)\preceq\type(\ulambda)\prec (\varnothing,(n))$}.
\]
\end{corollary}
\par
A $(\varnothing,\tau)$-transitive set is just a subset of $\GL(n,q)$ that is transitive on $\tau$-flags (where~$\tau$ is a partition of $n$). In this case Corollary~\ref{cor:designs_dual_dist} specialises to the following result, which is a perfect $q$-analog of~\cite[Theorem 4]{MarSag2007}.
\begin{corollary}
\label{cor:flag-transitive}
Let $Y$ be a subset of $\GL(n,q)$ with dual distribution $(a'_{\ulambda})$ and let~$\tau$ be a partition of $n$. Then $Y$ is $(\varnothing,\tau)$-transitive if and only if
\[
a'_{\ulambda}=0\quad\text{for all $\ulambda\in\Lambda_n$ satisfying $\tau\unlhd\ulambda(1)\lhd(n)$}.
\]
\end{corollary}
\par
Another immediate consequence of Theorem~\ref{thm:designs_dual_dist} is the following, of which Theorem~\ref{thm:main} arises as a special case.
\begin{corollary}
\label{cor:LVP-theorem}
Let $Y$ be a subset of $\GL(n,q)$ and suppose that $Y$ is $(\sigma,\tau)$-transitive for some $(\sigma,\tau)\in\Theta_{n,q}$. Then $Y$ is also $(\hat\sigma,\hat\tau)$-transitive for all $(\hat\sigma,\hat\tau)\in\Theta_{n,q}$ satisfying $(\sigma,\tau)\preceq(\hat\sigma,\hat\tau)$.
\end{corollary}
\par
In the remainder of this section we shall give a proof of Theorem~\ref{thm:designs_dual_dist}. A key step will be the following decomposition of the permutation character of $\alpha$-flags.
\begin{proposition}
\label{pro:decomposition_perm_char}
Let $\alpha\in\Sigma_{n,q}$, let $\xi$ be the permutation character of $\alpha$-flags, and let
\[
\xi=\sum_{\ulambda\in\Lambda_n}m_{\ulambda}\chi^{\ulambda}
\]
be the decomposition of $\xi$ into irreducible characters. Then we have
\[
m_{\ulambda}\ne0\quad\Leftrightarrow\quad \type(\alpha)\preceq\type(\ulambda).
\]
\end{proposition}
\begin{proof}
Write $\alpha=(\rho,\Ic)$, where $\rho=(\rho_1,\rho_2,\dots,\rho_k)$. We define a subgroup $H$ of $\GL(n,q)$ by 
\[
H=\left\{
\begin{bmatrix}
A_1 & *     & \cdots & *\\
       & A_2 & \cdots & *\\
       &        & \ddots & \vdots\\
       &        &            & A_k\\
\end{bmatrix}
:A_i\in\GL(\rho_i,q),\,\text{$A_i=I_{\rho_i}$ if $i\in\Ic$}.
\right\}.
\]
Then $H$ is the stabiliser of an $\alpha$-flag and we have 
\[
\xi=\Ind_H^{\GL(n,q)}(1_H),
\]
where $1_H$ is the trivial character of $H$. We first induce~$1_H$ to the parabolic subgroup~$P_\rho$. For $1\le i\le k$, let $\pi_i:P_\rho\to\GL(\rho_i,q)$ be the projections given in~\eqref{eqn:projections}. Hence
\[
1_H=\prod_{i=1}^k(1_i\circ\pi_i),
\]
where $1_i$ is the trivial character on the trivial subgroup of $\GL(\rho_i,q)$ for $i\in\Ic$ and~$1_i$ is the trivial character of $\GL(\rho_i,q)$ for $i\in\Jc$, where $\Jc$ is the complement of $\Ic$ in $\{1,2,\dots,k\}$. We have
\[
P_\rho/H\cong\prod_{i\in\Ic}\GL(\rho_i,q),
\]
as a direct product. By Frobenius reciprocity, for each $i\in\Ic$, the induction of $1_i$ to $\GL(\rho_i,q)$ equals
\[
\sum_{\ukappa\in\Lambda_{\rho_i}}f_{\ukappa}\,\chi^{\ukappa}
\]
(recall that $f_{\ukappa}$ is the degree of $\chi^{\ukappa}$). Hence we obtain
\[
\Ind_H^{P_\rho}(1_H)=\Bigg(\prod_{i\in\Jc}(1_i\circ\pi_i)\Bigg)\Bigg(\prod_{i\in\Ic}\sum_{\ukappa\in\Lambda_{\rho_i}}f_{\ukappa}\,(\chi^{\ukappa}\circ\pi_i)\Bigg).
\]
By the transitivity of induction, $\xi$ is obtained by inducing $\Ind_H^{P_\rho}(1_H)$ to $\GL(n,q)$. To determine the irreducible constituents of $\xi$, it is now enough to determine the irreducible constituents of the induced characters
\begin{equation}
\phi_1\cdot\phi_2\cdots\phi_k,   \label{eqn:prod_phi}
\end{equation}
where $\phi_i$ is an irreducible character of $\GL(\rho_i,q)$ for $i\in\Ic$ and $\phi_i$ is the trivial character of $\GL(\rho_i,q)$ for $i\in\Jc$. Since the product of characters is commutative, we may now assume without loss of generality that $\Ic=\{1,2,\dots,r\}$, where $r=\abs{\Ic}$. We put $\sigma=(\rho_1,\dots,\rho_r)$ and $\tau=(\rho_{r+1},\dots,\rho_k)$ (so that $(\sigma,\tau)$ is the type of $(\rho,\Ic)$). Now consider the parabolic subgroup $P=P_{(\rho_1,\dots,\rho_r,\abs{\tau})}$. We have
\[
P/P_\rho\cong\GL(\abs{\tau},q)/P_\tau
\]
and hence by Lemma~\ref{lem:Youngs_rule} the character~\eqref{eqn:prod_phi} induces on $P$ to
\[
\sum_{\nu\unrhd\tau}K_{\nu\tau}\chi^{X-1\mapsto\nu}\cdot\phi_1\cdots\phi_r.
\]
To determine the irreducible constituents of $\xi$, it is now enough to determine the irreducible constituents of the induced characters
\begin{equation}
\phi_0\cdot\phi_1\cdots\phi_r,   \label{eqn:induced_character_phi}
\end{equation}
where $\phi_0$ is a unipotent irreducible character of $\GL(\abs{\tau},q)$ corresponding to a partition~$\nu$ with $\nu\unrhd\tau$ and $\phi_i$ is an irreducible character of $\GL(\rho_i,q)$ for $1\le i\le r$.
\par
To prove the forward direction of the lemma, assume that $\chi^{\ulambda}$ is a constituent of some induced character of the form~\eqref{eqn:induced_character_phi}. Let $(\kappa,\lambda)$ be the type of~$\ulambda$ and let $(\kappa^{(i)},\lambda^{(i)})$ be the type of the element of $\Lambda_{\rho_i}$ indexing~$\phi_i$. Then Lemma~\ref{lem:LLR_general} implies that the parts of~$\kappa$ are exactly the parts of $\kappa^{(1)},\kappa^{(2)},\dots,\kappa^{(r)}$. Since $\phi_i$ is a character of $\GL(\rho_i,q)$, we find that $\kappa^{(i)}$ refines $(\rho_i)$ and hence~$\kappa$ refines~$\sigma$. By assumption there is some partition $\nu$ with $\abs{\nu}=\abs{\tau}$ such that $\nu\unrhd\tau$, which by Lemma~\ref{lem:LLR_general} satisfies $\nu\subseteq\lambda$. Hence we have $\lambda\unrhd\tau$, which proves the forward direction of the lemma.
\par
To prove the reverse direction, let $\ulambda\in\Lambda_n$ be such that its type $(\kappa,\lambda)$ satisfies $(\sigma,\tau)\preceq(\kappa,\lambda)$. Then~$\kappa$ is a refinement of $\sigma$ and $\tau\unlhd\lambda$. It is readily verified that there exists a partition $\nu$ with $\abs{\nu}=\abs{\tau}$ such that $\nu\unrhd\tau$ and $\nu\subseteq\lambda$. Let $\ulambda_0\in\Lambda_{\abs{\tau}}$ be given by $X-1\mapsto\nu$. Since $\kappa$ is a refinement of $\sigma$, there is a chain of partition-valued functions 
\[
\ulambda_0\subseteq \ulambda_1\subseteq\cdots\subseteq\ulambda_r=\ulambda
\]
with the property $\norm{\ulambda_i}-\norm{\ulambda_{i-1}}=\rho_i$ for all $i\in\{1,2,\dots,r\}$. 
By Remark~\ref{rem:LLR}, we can choose $\udelta_i\in\Lambda_{\rho_i}$ such that
\[
c_{\ulambda_{i-1},\udelta_i}^{\ulambda_i}>0\quad\text{for each $i\in\{1,2,\dots,r\}$}.
\]
Now we take $\phi_0=\chi^{\ulambda_0}$ and $\phi_i=\chi^{\udelta_i}$ for each $i\in\{1,2,\dots,r\}$. Then
\[
\phi_0\cdot\phi_1\cdots\phi_i
\]
has $\chi^{\ulambda_i}$ as an irreducible constituent for each $i\in\{1,2,\dots,r\}$. Hence $\chi^{\ulambda}$ is an irreducible constituent of $\phi_0\cdot\phi_1\cdots\phi_r$, which completes the proof.
\end{proof}
\par
Now, for $\alpha\in\Sigma_{n,q}$, let $\Fc_\alpha$ be the set of $\alpha$-flags and define $M_{\alpha}\in\CC(\GL(n,q),\Fc_\alpha\times\Fc_\alpha)$ to be the incidence matrix of elements of $\GL(n,q)$ versus left cosets of stabilisers of $\alpha$-flags by
\[
M_\alpha(g,(u,v))=\begin{cases}
1 & \text{for $gu=v$}\\
0 & \text{otherwise}.
\end{cases}
\]
Recall the definition of the vector spaces $V_{\ulambda}$ from Section~\ref{sec:association_schemes} and, for $(\sigma,\tau)\in\Theta_{n,q}$, define
\begin{equation}
U_{(\sigma,\tau)}=\sums{\ulambda\in\Lambda_n\\(\sigma,\tau)\preceq \type(\ulambda)}V_{\ulambda}.   \label{eqn:U_decomp}
\end{equation}
Note that in view of~\eqref{eqn:CC_Gn_decomp} this sum is direct.
\begin{corollary}
\label{cor:col_sp_M}
The column space of $M_{\alpha}$ equals $U_{\type(\alpha)}$.
\end{corollary}
\begin{proof}
Let $\xi$ be the permutation character of the set of $\alpha$-flags and define $C_\alpha\in\CC(\GL(n,q),\GL(n,q))$ by $C_\alpha(x,y)=\xi(x^{-1}y)$. Denoting by $\1_{xu=v}$ the indicator of the event that $x\in \GL(n,q)$ maps $u$ to $v$, we have
\begin{align*}
(M_\alpha M_\alpha^\top)(x,y)&=\sum_{u,v}M_\alpha(x,(u,v))M_\alpha(y,(u,v))\\
&=\sum_{u,v}\1_{xu=v}\1_{yu=v}\\
&=\sum_{u}\1_{xu=yu}\\
&=\sum_{u}\1_{x^{-1}yu=u}\\
&=\xi(x^{-1}y)=C_\alpha(x,y).
\end{align*}
Hence we have $C_\alpha=M_\alpha M_\alpha^\top$ and so the column space of $C_\alpha$ equals the column space of~$M_\alpha$.
\par
From Proposition~\ref{pro:decomposition_perm_char} and~\eqref{eqn:E_matrix_chi} we obtain 
\begin{equation}
C_\alpha=\abs{\GL(n,q)}\sums{\ulambda\in\Lambda_n\\\type(\alpha)\preceq\type(\ulambda)}(m_{\ulambda}/f_{\ulambda})E_{\ulambda}.   \label{eqn:C_t_from_E_lambda}
\end{equation}
Hence the column space of $C_\alpha$ is contained in $U_{\type(\alpha)}$. Conversely, let $v$ be a column of $E_{\ukappa}$ for some $\ukappa\in\Lambda_n$ satisfying $\type(\ukappa)\succeq \type(\alpha)$. From~\eqref{eqn:E_orthogonal} we have $E_{\ulambda}v=v$ for $\ukappa=\ulambda$ and $E_{\ulambda}v=0$ for $\ukappa\ne \ulambda$. Hence from~\eqref{eqn:C_t_from_E_lambda} we find that
\[
C_\alpha v=\abs{\GL(n,q)}\, (m_{\ukappa}/f_{\ukappa})\,v,
\]
and, since $m_{\ukappa}\ne 0$, we conclude that $v$ is in the column space of $C_\alpha$, as required.
\end{proof}
\par
We now complete the proof of Theorem~\ref{thm:designs_dual_dist}.
\begin{proof}[Proof of Theorem~\ref{thm:designs_dual_dist}]
Note that $Y$ is transitive on $\alpha$-flags if and only if
\[
\frac{1}{\abs{Y}}\,M_\alpha^\top\,\1_Y=\frac{1}{\abs{\GL(n,q)}}\,M_\alpha^\top\,\1_{\GL(n,q)},
\]
hence if and only if 
\[
\1_Y-\frac{\abs{Y}}{\abs{\GL(n,q)}}\1_{\GL(n,q)}
\]
is orthogonal to the column space of $M_\alpha$. In view of the orthogonal decomposition of this space given in Corollary~\ref{cor:col_sp_M} and the fact that $V_{X-1\mapsto (n)}$ is spanned by $1_{\GL(n,q)}$, we conclude that $Y$ is transitive on $\alpha$-flags if and only if $\1_Y$ is orthogonal to $V_{\ulambda}$ for each $\ulambda\in\Lambda_n$ satisfying $\type(\alpha)\preceq\type(\ulambda)\prec(\varnothing,(n))$. This is equivalent to the statement of the theorem.
\end{proof}


\section{Transitive subgroups}
\label{sec:subgroups}

In this section we classify subgroups $G$ of $\GL(n,q)$ that are $(\sigma,\tau)$-transitive. These results are essentially known. If $G$ is $((1),(n-1))$-transitive or $(\varnothing,(n-1,1))$-transitive or $((1^2),\varnothing)$ if $q=2$, then $G$ is transitive on $1$-spaces of $\FF_q^n$. Such subgroups have been classified by Hering~\cite{Her1974}, see also~\cite[Table~3.1]{GiuGlaPra2020} for a nice summary. However, as we always have examples of sharply $(\sigma,\tau)$-transitive subgroups in these cases (see Example~\ref{ex:Singer}), we shall henceforth assume that $(\sigma,\tau)$ is different from $((1),(n-1))$ and $(\varnothing,(n-1,1))$ and also different from $((1^2),\varnothing)$ if $q=2$. 
\par
First consider the case $n\ge 4$. By Corollary~\ref{cor:LVP-theorem}, $G$ is also $(\varnothing,(n-2,2))$-transitive, namely transitive on $2$-spaces of $\FF_q^n$. Kantor~\cite{Kan1973} proved that $G$ is either doubly transitive on $1$-spaces of $\FF_q^n$ or $G\cong\GammaL(1,2^5)$ as a subgroup of $\GL(5,2)$, which acts sharply transitive on $2$-spaces of $\FF_2^5$. Cameron and Kantor~\cite{CamKan1979} proved that, if $G$ is doubly transitive on $1$-spaces of $\FF_q^n$, then $G$ either contains $\SL(n,q)$, in which case $G$ is $((n-1),(1))$-transitive, or $G\cong A_7$ as a subgroup of $\GL(4,2)$. In fact it is computationally readily verified that the latter example is sharply $((31),\varnothing)$-transitive.
\par
Next consider the case $n=3$. Then by Corollary~\ref{cor:LVP-theorem}, $G$ is also $(\varnothing,(1^3))$-transitive when $q>2$ or $((1^3),\varnothing)$-transitive when $q=2$. That is, $G$ is transitive on $(1^3)$-flags in $\FF_q^3$, typically just called \emph{flags} in the literature. Kantor~\cite{Kan1987} proved that $G$ either contains $\SL(n,q)$ or $G$ acts sharply transitive on flags in $\FF_q^3$. Higman and McLaughlin~\cite{HigMcL1961} showed that in the latter case the only possibility is $G\cong\GammaL(1,2^3)$ as a subgroup of $\GL(3,2)$.
\par
Now consider the case $n=2$. Then we are left with the case that $G$ is $((1^2),\varnothing)$-transitive and $q>2$. The number of $((1^2),\varnothing)$-flags is $(q^2-1)(q-1)$ and the order of $G$ must be a multiple of this number. Since $\abs{\GL(2,q)}=(q^2-1)(q-1)q$, the index of $G$ in $\GL(2,q)$ must therefore be a divisor of $q$. Noting that $G$ is transitive on the $1$-spaces of $\FF_q^2$, an inspection of~\cite[Thm.~3.1]{GiuGlaPra2020} reveals that the only possible cases are $G\cong\GammaL(1,3^2)$ inside $\GL(2,3)$ or $q$ is one of the numbers $5,7,9,11,19,23,29,59$ and a computer verification reveals that only $\GL(2,3)$ and $\GL(2,5)$ contain subgroups $G$ in question. In the former case we have $G\cong\GammaL(1,3^2)$ and in the latter case $G$ is unique up to conjugation. In both cases $G$ is sharply $((1^2),\varnothing)$-transitive.
\par
We summarise these results in the following theorem.
\begin{theorem}
Suppose that $G$ is a $(\sigma,\tau)$-transitive nontrivial proper subgroup of $\GL(n,q)$ and $(\sigma,\tau)$ is different from $((1),(n-1))$ and $(\varnothing,(n-1,1))$ and also different from $((1^2),\varnothing)$ if $q=2$. Then one of the following holds:
\begin{enumerate}
\item $q>2$ and $G\ge\SL(n,q)$ and $G$ is $((n-1),(1))$-transitive.
\item $(n,q)=(2,3)$ and $G\cong\GammaL(1,3^2)$ is sharply $((1^2),\varnothing)$-transitive. 
\item $(n,q)=(2,5)$ and $G$ has order $96$ and 
is sharply $((1^2),\varnothing)$-transitive.
\item $(n,q)=(3,2)$ and $G\cong\GammaL(1,2^3)$ and $G$ is sharply $((1^3),\varnothing)$-transitive.
\item $(n,q)=(4,2)$ and $G\cong A_7$ is sharply $((31),\varnothing)$-transitive. 
\item $(n,q)=(5,2)$ and $G\cong\GammaL(1,2^5)$ is sharply $(\varnothing,(32))$-transitive.
\end{enumerate}
\end{theorem}
\par
It should be noted that there exist groups acting transitively on flags in $\FF_8^3$, namely~$\GammaL(1,2^9)$ and a subgroup of index $7$~\cite{HigMcL1961}. These groups however are not subgroups of~$\GL(3,8)$, but rather are subgroups of $\GammaL(3,8)$.


\section{Transitive sets and cliques}
\label{sec:Cliques}

In this section we consider so-called cliques in $\GL(n,q)$ and discuss their relationship to transitivity in $\GL(n,q)$.
\begin{definition}
Let $(\sigma,\tau)\in\Theta_{n,q}$. We define a subset $Y$ of $\GL(n,q)$ to be a \emph{$(\sigma,\tau)$-clique} if, for all distinct $x,y\in Y$, there is no $\alpha$-flag with $\type(\alpha)=(\sigma,\tau)$ fixed by~$x^{-1}y$.
\end{definition}
\par
For $\umu\in\Lambda$ we define $\umu'\in\Lambda$ to be the mapping $\umu':\Phi\to\Par$ given by $\umu'(f)=\umu(f)'$. Note that, if $\type(\umu)=(\nu,\mu)$, then we have $\type(\umu')=(\nu,\mu')$.
\par
The following result should be compared with Corollary~\ref{cor:designs_dual_dist}, showing that the concept of a $(\sigma,\tau)$-clique is dual to the concept of $(\sigma,\tau)$-transitivity.
\begin{theorem}
\label{thm:cliques_inner_dist}
Let $Y$ be a subset of $\GL(n,q)$ with inner distribution $(a_{\umu})$ and let $(\sigma,\tau)\in\Theta_{n,q}$. Then $Y$ is a $(\sigma,\tau)$-clique if and only if
\[
a_{\umu}=0\qquad\text{for all $\umu\in\Lambda_n$ satisfying $(\tau,\sigma)\preceq\type(\umu')\prec(\varnothing,(n))$}.
\]
\end{theorem}
\begin{proof}
Fix $\umu\in\Lambda_n$. Note that, for $\alpha\in\Sigma_{n,q}$, either all elements in $C_{\umu}$ fix an $\alpha$-flag or none of the elements in $C_{\umu}$. We show the elements in $C_{\umu}$ fix an $\alpha$-flag with $\type(\alpha)=(\sigma,\tau)$ if and only if $(\tau,\sigma)\preceq (\nu,\mu)$, where $(\nu,\mu)$ is the type of $\umu'$.
\par
First assume that $(\tau,\sigma)\preceq (\nu,\mu)$, namely $\sigma\unlhd\mu$ and $\nu$ refines $\tau$. Since $\sigma\unlhd\mu$, rearranging rows and columns of the Jordan canonical form of $C_{\umu}$ shows that $C_{\umu}$ contains a block upper-triangular matrix whose diagonal blocks are $I_{\sigma_1},I_{\sigma_2},\dots$ followed by $\abs{\mu}-\abs{\sigma}$ blocks of order $1$ followed by blocks of order $\nu_1,\nu_2,\dots$. Since $\nu$ refines $\tau$, this matrix fixes an $\alpha$-flag with $\type(\alpha)=(\sigma,\tau)$.
\par
Now let $g\in C_{\umu}$ be in Jordan canonial form and assume that $g$ fixes an $\alpha$-flag with $\type(\alpha)=(\sigma,\tau)$. By~\cite[Proposition~4.4]{LewReiSta2014} the companion matrix of an irreducible polynomial in $\FF_q[X]$ of degree $d$ does not fix a proper subspace of $\FF_q^d$. Hence $\nu$ must refine $\tau$. Also note that $g$ has $\mu_i$ Jordan blocks with eigenvalue $1$ of order at least $i$ and each such Jordan block of order $i$ fixes a $\beta$-flag with $\type(\beta)=((1^i),\varnothing)$. Hence~$g$ must have at least 
\[
\sigma_i-\sum_{j=1}^{i-1}(\mu_j-\sigma_j)
\]
Jordan blocks with eigenvalue $1$ of order at least $i$. The latter statement is equivalent to $\sigma\unlhd\mu$.
\end{proof}
\par
In what follows we establish relationships between $(\sigma,\tau)$-cliques and $(\sigma,\tau)$-transitive sets in $\GL(n,q)$. 
\begin{theorem}
\label{thm:bounds_cliques_designs}
Let $(\sigma,\tau)\in\Theta_{n,q}$, let $H$ be the stabiliser of an $\alpha$-flag with $\type(\alpha)=(\sigma,\tau)$, and let $Y$ be a subset of $\GL(n,q)$.
\begin{enumerate}[(1)]
\item If $Y$ is a $(\sigma,\tau)$-clique, then $\abs{Y}\le \abs{\GL(n,q)}/\abs{H}$ with equality if and only if $Y$ is $(\sigma,\tau)$-transitive.
\item If $Y$ is $(\sigma,\tau)$-transitive, then $\abs{Y}\ge \abs{\GL(n,q)}/\abs{H}$ with equality if and only if~$Y$ is a $(\sigma,\tau)$-clique.
\end{enumerate}
In both cases, equality implies that $Y$ is sharply $(\sigma,\tau)$-transitive. 
\end{theorem}
\begin{proof}
Since, for each $(x,y)\in H\times Y$, there is a unique $g\in\GL(n,q)$ such that $gx=y$, we have
\begin{equation}
\sum_{g\in\GL(n,q)}\abs{Y\cap gH}=\abs{Y}\cdot\abs{H}.   \label{eqn:clique_identity}
\end{equation}
The quotient of any two distinct elements in $Y\cap gH$ fixes an $\alpha$-flag of type $(\sigma,\tau)$. Hence, if $Y$ is a $(\sigma,\tau)$-clique, then each summand on the left hand side of~\eqref{eqn:clique_identity} is at most $1$, which gives the bound in (1). If $H$ is the stabiliser of the $\alpha$-flag $F$, then $gH$ contains precisely all elements of $\GL(n,q)$ mapping $F$ to $gF$. Hence, if $Y$ is $(\sigma,\tau)$-transitive, then each summand on the left hand side of~\eqref{eqn:clique_identity} must be at least $1$, which gives the bound in (2).
In both cases, equality occurs if and only if $\abs{Y\cap gH'}=1$ for each $g\in G$ and the stabiliser $H'$ of each $\alpha$-flag of type $(\sigma,\tau)$. By the same reasoning as above, this establishes the characterisations of equality.
\end{proof}
\par
Note that, if $H$ is the stabiliser of an $\alpha$-flag with $\type(\alpha)=(\sigma,\tau)$, then an elementary counting argument gives
\[
\frac{\abs{\GL(n,q)}}{\abs{H}}=\frac{[n]_q!}{\big(\prod_{i\ge 1}[\sigma_i]_q!\big)\big(\prod_{i\ge 1} [\tau_i]_q!\big)}\prod_{i\ge 1}\prod_{j=1}^{\sigma_i-1}(q^{\sigma_i}-q^j),
\]
where, for a nonnegative integer $m$, the \emph{$q$-factorial} of $m$ is given by
\[
[m]_q!=[m]_q[m-1]_q\cdots[1]_q\quad\text{with $[k]_q=\frac{q^k-1}{q-1}$}.
\]
\par
In view of Theorems~\ref{thm:cliques_inner_dist} and~\ref{thm:bounds_cliques_designs} one can rule out the existence of sharply $(\sigma,\tau)$-transitive subsets of $\GL(n,q)$ by linear programming, a standard method in the theory of association schemes. The key observation is that the entries in the dual distribution of a subset of an association scheme are real and nonnegative (see Section~\ref{sec:association_schemes}). The so-called linear-programming (LP) bound for $(\sigma,\tau)$-cliques is the maximum of
\[
\sum_{\umu\in\Lambda_n}a_{\umu}
\]
subject to the constraints
\begin{gather*}
a_{\umu}\ge 0\quad\text{for all $\umu\in\Lambda_n$},\\[1ex]
\sum_{\umu\in\Lambda_n}\Im(\chi^{\ulambda}_{\umu})\;a_{\umu}=0\quad\text{and}\quad \sum_{\umu\in\Lambda_n}\Re(\chi^{\ulambda}_{\umu})\;a_{\umu}\ge 0\quad\text{for all $\ulambda\in\Lambda_n$},\\[1ex]
a_{\umu}=0\quad\text{for all $\umu\in\Lambda_n$ satisfying $(\tau,\sigma)\preceq\type(\umu')\prec(\varnothing,(n))$}.
\end{gather*}
Here the second constraint comes from the fact that the entries in the dual distribution of a subset of $\GL(n,q)$ are real and nonnegative. We have determined the LP bound for $(\sigma,\tau)$-cliques in $\GL(n,2)$ for $n\in\{2,3,4,5\}$. The LP bound coincides with the bound of Theorem~\ref{thm:bounds_cliques_designs}~(i) except for those pairs $(\sigma,\tau)$ shown in Table~\ref{tab:LP_bounds}. Consequently no sharply $(\sigma,\tau)$-transitive subsets of $\GL(n,q)$ can exist in these cases.
\par
\begin{table}[ht]
\caption{Bounds for cliques in $\GL(4,2)$ and $\GL(5,2)$.}
\label{tab:LP_bounds}
\centering
\begin{tabular}{c|c|c}
\hline
$(\sigma,\tau)$ & \multicolumn{1}{|c|}{Bound of Thm.~\ref{thm:bounds_cliques_designs}} & \multicolumn{1}{|c}{LP bound}\\\hline\hline
$((21^2),\varnothing)$ & 630 & 420\\
$((1^2),(2))$ & 105 & 84\\
$((2),(2))$ & 210 & 168\\\hline
$((32),\varnothing)$ & 156\,240 & 139\,500\\
$((31^2),\varnothing)$ & 78\,120 & 53\,010\\
$((221),\varnothing)$ & 39\,060 & 24\,180\\
$((21^3),\varnothing)$ & 19\,530 & 11\,718\\
$((3),(2))$ & 26\,040 & 19\,530\\
$((21),(2))$ & 6\,510 & 3\,550\\
$((1^3),(2))$ & 3\,255 & 2\,604\\
$((1),(22))$ & 1\,085 & 805\\\hline
\end{tabular}
\end{table}


\section{Existence results}
\label{sec:Existence}

In this section we show that, for a partition $\sigma$, nonnegative integers $\tau_2\ge\tau_3\ge \cdots$, and sufficiently large $n$, there exist $(\sigma,\tau)$-transitive sets in $\GL(n,q)$ that are arbitrarily small compared to $\GL(n,q)$, where $\tau_1=n-\abs{\sigma}-\tau_2-\tau_2-\cdots$. In view of Corollary~\ref{cor:LVP-theorem}, it suffices to consider $((t),(n-t))$-transitive sets in $\GL(n,q)$. For brevity, we shall call such a set a \emph{$t$-design} in $\GL(n,q)$. These objects will be studied in more detail in Section~\ref{sec:designs_codes_polynomials}.
\par
We give a recursive construction of $t$-designs in $\GL(n,q)$ using $t$-designs in the Grassmannian $J_q(n,k)$, namely the collection of all $k$-spaces of $\FF_q^n$. A \emph{$t$-design} in $J_q(n,k)$ is a subset $D$ of $J_q(n,k)$ such that the number of elements in $D$ containing a given $t$-space of $\FF_q^n$ is independent of the particular choice of this $t$-space. Our construction can be understood as a $q$-analog of the construction given in~\cite[Section~6]{MarSag2007} for the symmetric group $\Sym(n)$.
\par
Let $V=\FF_q^n$ and, for a $k$-space $U$ of $V$, let $\GL(U)$ be the general linear group of~$U$, which is of course isomorphic to $\GL(k,q)$. Fix a $k$-space $U$ of~$V$ and an $(n-k)$-space~$W$ of $V$ such that
\[
V=U\oplus W.
\]
For our recursive construction, we need three ingredients: a $t$-design $Y$ in $\GL(U)$, a $t$-design $Z$ in $\GL(W)$, and a $t$-design $D$ in $J_q(n,k)$. For each $B\in D$, there are $q^{k(n-k)}$ complementary spaces, namely $(n-k)$-spaces~$C$ with $V=B\oplus C$. We denote the collection of such spaces by $C_B$. For each~$B\in D$, we fix an isomorphism $g_B:U\to B$ and, for each $B\in D$ and each $C\in C_B$, we fix an isomorphism $h_{B,C}:W\to C$.
\par
Note that, given a pair $(B,C)$ with $B\in D$ and $C\in C_B$, then every pair of isomorphisms $(y,z)$ on $B$ and $C$ can be uniquely extended to an isomorphism on $V$ by linearity. We denote this extension by $(y,z)$. Hence, if $v\in V$, then there are unique $b\in B$ and $c\in C$ with $v=b+c$ and we have
\[
(y,z)(v)=y(b)+z(c).
\]
The following lemma contains a recursive construction of $t$-designs in $\GL(n,q)$.
\begin{lemma}
\label{lem:recursive_construction}
Let $Y$ be a $t$-design in $\GL(U)$, let $Z$ be a $t$-design in $\GL(W)$, and let~$D$ be a $t$-design in $J_q(n,k)$. Then the set
\begin{equation}
\{(g_B\circ y,h_{B,C}\circ z):y\in Y,z\in Z,B\in D,C\in C_B\}   \label{eqn:recursive_design}
\end{equation}
is a $t$-design in $\GL(V)$.
\end{lemma}
\par
Note that, taking $Y=\GL(U)$, $Z=\GL(W)$, and $D=J_q(n,k)$, the set constructed in Lemma~\ref{lem:recursive_construction} equals $\GL(V)$.
\begin{example}
By~\cite{BraKerLau2005} there exists a $2$-design in $J_2(6,3)$ of cardinality $279$. Taking $Y$ and $Z$ to be isomorphic to $\GL(3,2)$ in Lemma~\ref{lem:recursive_construction}, we obtain a $2$-design in $\GL(6,2)$ of cardinality $\tfrac{1}{5}\abs{\GL(6,2)}$. 
\end{example}
\par
To prove Lemma~\ref{lem:recursive_construction}, we shall need the following well known result about designs in $J_q(n,k)$, in which
\[
{n\brack k}_q=\frac{[n]_q!}{[k]_q!\,[n-k]_q!}
\]
is the \emph{$q$-binomial coefficient} counting the number of $k$-spaces of $\FF_q^n$.
\begin{lemma}[{\cite[Lemma 2.1]{Suz1990},~\cite[Fact~1.5]{KiePav2015}}]
\label{lem:intersection_numbers}
Let $D$ be a $t$-design in $J_q(n,k)$ and let $i,j$ be nonnegative integers satisfying $i+j\le t$. Let $I$ be an $i$-space of $V$ and let~$J$ be a $j$-space of $V$ such that $I\cap J=\{0\}$. Then the number
\[
m_{i,j}=\abs{\{B\in D:I\le B\wedge B\cap J=\{0\}\}}
\]
is independent of the particular choice of $I$ and $J$ and given by
\[
m_{i,j}=\abs{D}\,q^{j(k-i)}\frac{{n-i-j\brack k-i}_q{k\brack t}_q}{{n-t\brack k-t}_q{n\brack t}_q}.
\]
\end{lemma}
\par
We are now ready to prove Lemma~\ref{lem:recursive_construction}.
\begin{proof}[Proof of Lemma~\ref{lem:recursive_construction}]
\par
Choose $t$-tuples $(v_1,v_2,\dots,v_t)$ and $(v'_1,v'_2,\dots,v'_t)$ of linearly independent vectors of $V$. Suppose that exactly $i$ of the vectors $v_1,v_2,\dots,v_t$ are in~$U$. After reordering we can assume that these are $v_1,v_2,\dots,v_i$. Then the remaining $j=t-i$ vectors $v_{i+1},v_{i+2},\dots, v_t$ are outside $U$, namely they belong to complementary spaces of $U$. 
\par
The number of elements $B\in D$ containing $v'_1,v'_2,\dots,v'_i$, but none of the vectors $v'_{i+1},v'_{i+2},\dots, v'_t$, equals the constant $m_{i,j}$ given in Lemma~\ref{lem:intersection_numbers} and, for each such~$B$, there are $q^{k(n-k-j)}$ complementary spaces $C\in C_B$ containing the remaining $j$ vectors. Fix a pair $(C,B)$ with these properties. Write $v_\ell=u_\ell+w_\ell$ with $u_\ell\in U$ and $w_\ell\in W$ for all $\ell$ and note that our assumption implies that $v_\ell=u_\ell$ for all $\ell\le i$. There is a constant~$r_i$ such that there are exactly $r_i$ elements $y\in Y$ taking $v_\ell$ to $g_B^{-1}(v'_\ell)$ for all $\ell\le i$. For each such $y\in Y$, there is a constant $s_j$ such that there are exactly $s_j$ elements $z\in Z$ taking $w_\ell$ to
\[
h_{B,C}^{-1}(v'_\ell-g_B(y(u_\ell)))
\]
for all $\ell>i$.
\par
Hence the total number of automorphisms in~\eqref{eqn:recursive_design} taking the tuple $(v_1,v_2,\dots,v_t)$ to the tuple $(v'_1,v'_2,\dots,v'_t)$ equals
\[
m_{i,j}\,r_i\,s_j\,q^{k(n-k-j)}.
\]
We have to show that this number is independent of $i$. Lemma~\ref{lem:intersection_numbers} implies that 
\[
(q^k-q^i)\,m_{i,j}=(q^n-q^{k+j-1})\,m_{i+1,j-1}
\]
and it is readily verified that
\[
r_i=(q^k-q^i)\,r_{i+1}
\]
for $i\le t-1$ and
\[
s_j=(q^{n-k}-q^j)\,s_{j+1}
\]
for $j\le t-1$. By combining these identities we find that
\[
m_{i+1,j-1}\,r_{i+1}\,s_{j-1}\,q^{k(n-k-j+1)}=m_{i,j}\,r_i\,s_j\,q^{k(n-k-j)},
\]
which completes the proof.
\end{proof}
\par
The following existence result for $t$-designs in $J_q(n,k)$ was obtained by Fazeli, Lovett, and Vardy~\cite{FLV2014}, using the probabilistic approach of Kuperberg, Lovett, and Peled~\cite{KLP2017}.
\begin{lemma}
\label{lem:existence_subspace_designs}
If $k>12(t+1)$ and $n\ge ckt$ for some universal constant $c$, then there exists a $t$-design in $J_q(n,k)$ of cardinality at most $q^{12(t+1)n}$. 
\end{lemma}
\par
We now use the recursive construction in Lemma~\ref{lem:recursive_construction} together with Lemma \ref{lem:existence_subspace_designs} to obtain the following existence result for $t$-designs in $\GL(n,q)$.
\begin{theorem}
\label{thm:existence_designs}
Let $t$ be a positive integer and let $\e>0$. Then, for all sufficiently large $n$, there exists a $t$-design $Y$ in $\GL(n,q)$ satisfying $\abs{Y}/\abs{\GL(n,q)}<\e$.
\end{theorem}
\begin{proof}
Fix $k>12(t+1)$. We apply Lemma~\ref{lem:recursive_construction} with $Y=\GL(U)$ and $Z=\GL(W)$. Then from Lemma~\ref{lem:existence_subspace_designs} we obtain the existence of a $t$-design in $\GL(n,q)$ of cardinality at most
\[
N=\abs{\GL(k,q)}\cdot\abs{\GL(n-k,q)}\,q^{k(n-k)}q^{12(t+1)n},
\]
provided that $n\ge ckt$ for the constant $c$ of Lemma~\ref{lem:existence_subspace_designs}. Note that we have
\[
\frac{N}{\abs{\GL(n,q)}}=\frac{q^{12(t+1)n}}{{n\brack k}_q}<\frac{q^{12(t+1)n}}{q^{k(n-k)}}.
\]
Since $k>12(t+1)$, this number tends to zero as $n$ tends to infinity.
\end{proof}
\par
By combining Theorem~\ref{thm:existence_designs} and Corollary~\ref{cor:LVP-theorem} we obtain an existence result for general $(\sigma,\tau)$-transitive sets in $\GL(n,q)$. 
\begin{corollary}
Let $(\sigma,\tilde\tau)\in\Theta_{t,q}$ and let $\e>0$. Then for all sufficiently large $n$, there exists a $(\sigma,\tau)$-transitive set $Y$ in $\GL(n,q)$ satisfying $\abs{Y}/\abs{\GL(n,q)}<\e$, where $\tau=(n-\abs{\sigma}-\abs{\tilde\tau},\tilde\tau_1,\tilde\tau_2,\dots)$.
\end{corollary}


\section{Designs, codes, and orthogonal polynomials}
\label{sec:designs_codes_polynomials}

Certain association schemes, namely $P$- and $Q$-polynomial association schemes, are closely related to orthogonal polynomials in the sense that their character tables arise as evaluations of such polynomials (see~\cite{BanIto1984} or~\cite{Del1973}, for example). The conjugacy class association scheme of $\GL(n,q)$ does not have these properties. Nevertheless there is still a relationship to certain orthogonal polynomials, namely the Al-Salam-Carlitz polynomials. 
\par
We shall first recall and establish some basic properties of these polynomials and then apply these results to subsets of $\GL(n,q)$.

\subsection{Al-Salam-Carlitz polynomials}

The Al-Salam-Carlitz polynomials are given by
\[
U_k^{(a)}(x)=\sum_{j=0}^k(-1)^{k-j}q^{k-j\choose 2}{k\brack j}_q\prod_{i=0}^{j-1}(x-aq^i).
\]
They were introduced in~\cite{AlSCar1965} and some properties can be found in~\cite{Chi1978} and~\cite{Kim1997}. We are only interested in the case $a=1$ and write $U_k(x)$ for $U_k^{(1)}(x)$. These polynomials satisfy the recurrence relation
\[
U_{k+1}(x)=(x-2q^k)U_k(x)+q^{k-1}(1-q^k)U_{k-1}(x)\quad\text{for $k\ge 0$}
\]
with the initial condition $U_{-1}(x)=0$ and $U_0(x)=1$. The first polynomials are
\begin{align*}
U_1(x)&=x-2\\
U_2(x)&=x^2-2(q+1)x+3q+1\\
U_3(x)&=x^3-2(q^2+q+1)x^2+(3q^3+4q^2+4q+1)x-2q(2q^2+q+1).
\end{align*}
An equivalent definition of the Al-Salam-Carlitz polynomials is 
\begin{equation}
\label{eqn:AlSalam_Carlitz_moments}
\sum_{k=0}^j{j\brack k}_qU_k(x)=\prod_{i=0}^{j-1}(x-q^i)\quad\text{for $j=0,1,\dots$}.
\end{equation}
This follows from the inversion formula
\begin{equation}
\sum_{k=j}^\ell(-1)^{k-j}q^{k-j\choose 2}{k\brack j}_q{\ell\brack k}_q=\delta_{j\ell},   \label{eqn:inversion_formula}
\end{equation}
which in turn can be obtained from the $q$-binomial theorem.
\par
The Al-Salam-Carlitz polynomials are $q$-analogs of the Charlier polynomials and are orthogonal with respect to a $q$-analog of a Poisson distribution, whose $k$-th moment is 
\begin{equation}
\label{eqn:moment_q_Poisson}
\sum_{i=0}^k{k\brack i}_q,
\end{equation}
the number of subspaces of a $k$-dimensional vector space over~$\FF_q$. Let $\theta$ denote the class function of $\GL(n,q)$ given by
\[
\theta(g)=q^{n-\rk(g-I)}
\]
for each $g\in\GL(n,q)$, where $I$ is the identity of $\GL(n,q)$. Let $w_i$ be the number of elements $g\in\GL(n,q)$ satisfying $\theta(g)=q^i$. Explicit expressions for $w_i$ were obtained by Rudvalis and Shinoda in an unpublished work~\cite{RudShi1988} and by Fulman~\cite{Ful1999}, which shows that
\begin{equation}
w_i=\frac{\abs{\GL(n,q)}}{\abs{\GL(i,q)}}\sum_{k=0}^{n-i}\frac{(-1)^kq^{k\choose 2}}{q^{ki}\,\abs{\GL(k,q)}}.   \label{eqn:wi}
\end{equation}
We shall later see that this expression also follows from our results (see Remark~\ref{rem:Rudvalis_Shinoda}).
\par
The class function $\theta$ defines a discrete random variable on $\GL(n,q)$ and it was shown in~\cite{FulSta2016} that its $k$-th moment equals~\eqref{eqn:moment_q_Poisson}, provided that $k\le n$. Hence the Al-Salam-Carlitz polynomials also satisfy the orthogonality relation
\begin{equation}
\label{eqn:orthogonality_AlSalam_Carlitz}
\sum_{i=0}^nw_i\,U_k(q^i)U_\ell(q^i)=0\quad\text{for $k\ne \ell$ and $k+\ell\le n$}.
\end{equation}
(It follows from Theorem~\ref{thm:decomposition_AlSC_chars} that,  for $k=\ell$ and $2k\le n$, the evaluation of the left-hand side is $\abs{\GL(k,q)}\cdot\abs{\GL(n,q)}$.)
\par
With every polynomial $f(x)=f_nx^n+\cdots+f_1x+f_0$ in $\RR[x]$ we associate the class function $f(\theta)=f_n\theta^n+\cdots+f_1\theta+f_0$. This induces an algebra homomorphism from~$\RR[x]$ to the set of class functions of $\GL(n,q)$. Let~$\xi_j$ be the permutation character on ordered $j$-tuples of linearly independent elements of $\FF_q^n$. By convention $\xi_0$ is the trivial character of $\GL(n,q)$. Note that
\[
\xi_j=\prod_{i=0}^{j-1}(\theta-q^i).
\]
Hence we have
\begin{equation}
U_k(\theta)=\sum_{j=0}^k(-1)^{k-j}q^{k-j\choose 2}{k\brack j}_q\xi_j\quad\text{for $k=0,1,\dots,n$}   \label{eqn:U_from_xi}
\end{equation}
and by~\eqref{eqn:AlSalam_Carlitz_moments}
\begin{equation}
\xi_j=\sum_{k=0}^j{j\brack k}_qU_k(\theta)\quad\text{for $j=0,1,\dots,n$}.   \label{eqn:xi_from_U}
\end{equation}
For $0\le k\le n/2$, we now decompose $U_k(\theta)$ into irreducible characters of~$\GL(n,q)$.
\begin{theorem}
\label{thm:decomposition_AlSC_chars}
For $0\le k\le n/2$, the decomposition of $U_k(\theta)$ into irreducible characters is 
\[
U_k(\theta)=\sum_{\unu\in\Lambda_k}f_{\unu}\,\chi^{r(\unu)},
\]
where $r(\unu)$ is the element $\ulambda\in\Lambda_n$ that agrees with $\unu$ except on $1$, where it is $\ulambda(1)=(n-k,\unu(1)_1,\unu(1)_2,\dots)$, namely $\ulambda(1)$ is obtained from $\unu(1)$ by inserting a row with $n-k$ cells in the Young diagram of $\unu(1)$.
\end{theorem}
\begin{proof}
Recall that the standard scalar product on class functions $\phi$ and $\psi$ of $\GL(n,q)$ is given by
\[
\ang{\phi,\psi}=\frac{1}{\abs{\GL(n,q)}}\sum_{g\in\GL(n,q)}\phi(g)\overline{\psi(g)}.
\]
It follows from the orthogonality relation~\eqref{eqn:orthogonality_AlSalam_Carlitz} that
\[
\ang{U_k(\theta),U_\ell(\theta)}=0\quad\text{for $0\le k<\ell\le n/2$}.
\]
Since the irreducible characters of $\GL(n,q)$ form an orthonormal basis for the space of class functions of $\GL(n,q)$, we find from~\eqref{eqn:xi_from_U} that $U_k(\theta)$ decomposes into those irreducible characters that occur in the decomposition of $\xi_k$, but not in the decomposition of $\xi_{k-1}$.
\par
As in the proof of Proposition~\ref{pro:decomposition_perm_char} we have
\[
\xi_k=\sum_{\unu\in\Lambda_k}f_{\unu}\;(\chi^{\unu}\cdot 1_{\GL(n-k,q)}),
\]
where $1_{\GL(n-k,q)}$ is the trivial character of $\GL(n-k,q)$. Note that the Littlewood-Richardson coefficient $c_{\nu,(n-k)}^\mu$ is either $0$ or $1$ and it equals $1$ precisely when the Young diagram of $\mu$ is obtained from that of $\nu$ by adding $n-k$ cells no two of which are in the same column (this special case is also known as Pieri's rule). Hence by Lemma~\ref{lem:LLR_general} the character $\chi^{\unu}\cdot 1_{\GL(n-k,q)}$ decomposes into those irreducible characters $\chi^{\ulambda}$ for which $\ulambda$ agrees with $\unu$ except on~$1$ and~$\ulambda(1)$ is obtained from $\unu(1)$ by adding $n-k$ cells to the Young diagram of $\unu(1)$ no two of which in the same column. Hence the irreducible characters occuring in the decomposition of $\xi_k$ but not in the decomposition of $\xi_{k-1}$ are precisely $\chi^{r(\unu)}$ with multiplicity $f_{\unu}$, where $\unu\in\Lambda_k$.
\end{proof}
\par
By combining Theorem~\ref{thm:decomposition_AlSC_chars} and~\eqref{eqn:xi_from_U}, we obtain the decomposition into irreducible characters of $\xi_j$ for $0\le j\le n/2$. This result strengthens Proposition~\ref{pro:decomposition_perm_char} for $(\sigma,\tau)=((t),(n-t))$ and $t\le n/2$.
\begin{corollary}
For $0\le j\le n/2$ the decomposition of $\xi_j$ into irreducible characters is 
\[
\xi_j=\sum_{k=0}^j{j\brack k}_q\,\sum_{\unu\in\Lambda_k}f_{\unu}\,\chi^{r(\unu)},
\]
where $r(\unu)$ is as in Theorem~\ref{thm:decomposition_AlSC_chars}.
\end{corollary}

\subsection{Designs and codes}

Henceforth we call a $((t),(n-t))$-transitive subset of $\GL(n,q)$ a $t$-\emph{design}. Thus a $t$-design in $\GL(n,q)$ is transitive on the set of $t$-tuples of linearly independent elements of $\FF_q^n$. We also call an $((n-d+1),(d-1))$-clique a \emph{$d$-code}. Hence, for all distinct elements $x,y$ of a $d$-code, there is no $(n-d+1)$-tuple of linearly independent elements of $\FF_q^n$ fixed by $x^{-1}y$. This implies that $\rk(x-y)\ge d$ for all distinct $x,y$ in a $d$-code. 
\par
Theorems~\ref{thm:designs_dual_dist} and~\ref{thm:cliques_inner_dist} specialise in these cases as follows.
\begin{corollary}
\label{cor:designs_codes_zero_distributions}
Let $Y$ be a subset of $\GL(n,q)$ with inner distribution $(a_{\umu})$ and dual distribution $(a'_{\ulambda})$. Then $Y$ is a $t$-design if and only if
\[
a'_{\ulambda}=0\quad\text{for each $\ulambda\in\Lambda_n$ satisfying $n-t\le \ulambda(1)_1<n$}
\]
and a $d$-code if and only if
\[
a_{\umu}=0\quad\text{for each $\umu\in\Lambda_n$ satisfying $n-d+1\le\umu(1)'_1<n$}
\]
\end{corollary}
\par
Note that the mapping $(x,y)\mapsto \rk(x-y)$ is a metric on $\GL(n,q)$. Accordingly, for a subset $Y$ of $\GL(n,q)$, we define the \emph{distance distribution} to be the tuple $(A_i)_{0\le i\le n}$, where
\[
A_i=\frac{1}{\abs{Y}}\bigabs{\{(x,y)\in Y\times Y:\rk(x-y)=i\}}
\]
and the \emph{dual distance distribution} to be the tuple $(A'_k)_{0\le k\le n}$, where
\[
A'_k=\sum_{i=0}^nU_k(q^{n-i})A_i.
\]
Note that
\begin{equation}
A'_k=\frac{1}{\abs{Y}}\sum_{x,y\in Y}U_k(q^{n-\rk(x-y)}).   \label{eqn:Dual_dist_U}
\end{equation}
We now characterise $t$-designs in terms of zeros in its dual distance distribution.
\begin{proposition}
\label{pro:t-designs_dual_dist}
Let $Y$ be a subset of $\GL(n,q)$ with dual distance distribution $(A'_k)$ and let $t$ be an integer satisfying $1\le t\le n$. If $Y$ is a $t$-design, then $A'_k=0$ for all $k$ satisfying $1\le k\le t$. Moreover the converse also holds if $t\le n/2$. That is, if $t\le n/2$ and $A'_k=0$ for all $k$ satisfying $1\le k\le t$, then $Y$ is a $t$-design.
\end{proposition}
\begin{proof}
First suppose that $Y$ is a $t$-design. From~\eqref{eqn:Dual_dist_U} and~\eqref{eqn:U_from_xi} we have
\begin{equation}
A'_k=\frac{1}{\abs{Y}}\sum_{j=0}^k(-1)^{k-j}q^{k-j\choose 2}{k\brack j}_q\sum_{x,y\in Y}\xi_j(x^{-1}y)   \label{eqn:A_k_xi}.
\end{equation}
By Proposition~\ref{pro:decomposition_perm_char}, the permutation character $\xi_j$ decomposes into those irreducible characters $\chi^{\ulambda}$ for which $((j),(n-j))\preceq\type(\ulambda)$ or equivalently $\ulambda(1)_1\ge n-j$. Moreover, since $\xi_j$ is a permutation character, it contains the trivial character with multiplicity $1$. From Corollary~\ref{cor:designs_codes_zero_distributions} we then find that the inner sum in~\eqref{eqn:A_k_xi} equals $\abs{Y}^2$ for all $j$ satisfying $0\le j\le t$. Hence we have, for all $k$ satisfying $0\le k\le t$,
\[
A'_k=\abs{Y}\sum_{j=0}^k(-1)^{k-j}q^{k-j\choose 2}{k\brack j}_q=\abs{Y}\,\delta_{k,0},
\]
using~\eqref{eqn:inversion_formula} together with elementary manipulations.
\par
Now, for each $k$ satisfying $0\le k\le n/2$, we find from~\eqref{eqn:Dual_dist_U}, Theorem~\ref{thm:decomposition_AlSC_chars}, and~\eqref{eqn:dual_dist_explicit} that
\begin{align*}
A'_k&=\frac{1}{\abs{Y}}\sum_{\unu\in\Lambda_k}f_{\unu}\sum_{x,y\in Y}\chi^{r(\unu)}(x^{-1}y)\\
&=\sum_{\unu\in\Lambda_k}\frac{f_{\unu}}{f_{r(\unu)}}\,a'_{r(\unu)},
\end{align*}
where $r(\unu)$ is as in Theorem~\ref{thm:decomposition_AlSC_chars}. Suppose that $t$ satisfies $1\le t\le n/2$ and that $A'_k=0$ for all $k$ satisfying $1\le k\le t$. Since $f_{\unu}/f_{r(\unu)}$ is positive, we find that $a'_{r(\unu)}=0$ for all $\unu\in\Lambda_k$ and hence $a'_{\ulambda}=0$ for all $\ulambda\in\Lambda_n$ satisfying $n-t\le \ulambda(1)_1<n$. Corollary~\ref{cor:designs_codes_zero_distributions} then implies that $Y$ is a $t$-design.
\end{proof}
\par
Theorem~\ref{thm:bounds_cliques_designs} specialises as follows.
\begin{corollary}
\label{cor:bound_codes}
Let $Y$ be a subset of $\GL(n,q)$ and let $d$ and $t$ be the largest integers such that $Y$ is a $d$-code and a $t$-design. Then
\[
\prod_{i=0}^{t-1}(q^n-q^i)\le\abs{Y}\le \prod_{i=0}^{n-d}(q^n-q^i).
\]
Moreover, if equality holds in one of the bounds, then equality also holds in the other and this case happens if and only if $d=n-t+1$.
\end{corollary}
\par
The upper bound in Corollary~\ref{cor:bound_codes} is a $q$-analog of a corresponding well known bound $n(n-1)\cdots d$ for permutation codes~\cite{BlaCohDez1979}. The bounds in Corollary~\ref{cor:bound_codes} can be achieved. A Singer cycle in $\GL(n,q)$ 
gives an $n$-code in $\GL(n,q)$ of size $q^n-1$ (see Example~\ref{ex:Singer}) and $A_7$ inside $\GL(4,2)$ is a $2$-code of size $2520$ (see Section~\ref{sec:subgroups}).
\par
It turns out that the distance distribution of a subset $Y$ of $\GL(n,q)$ is uniquely determined by its parameters, provided that $Y$ is a $t$-design and a $d$-code, where $d\ge n-t$. The following result generalises~\eqref{eqn:wi}.
\begin{theorem}
\label{thm:inner_dist}
Suppose that $Y$ is a $t$-design and an $(n-t)$-code in $\GL(n,q)$. Then the distance distribution $(A_i)$ of $Y$ satisfies
\[
A_{n-i}=\sum_{j=i}^t(-1)^{j-i}q^{j-i\choose 2}{j\brack i}_q{n\brack j}_q\bigg(\frac{\abs{Y}}{\prod_{k=0}^{j-1}(q^n-q^k)}-1\bigg)
\]
for each $i\in\{0,1,\dots,n-1\}$.
\end{theorem}
\begin{proof}
We have
\[
A'_k=\sum_{i=0}^nU_k(q^i)A_{n-i}.
\]
Multiply both sides by ${j\brack k}_q$, sum over $k$, and use~\eqref{eqn:AlSalam_Carlitz_moments} to find that
\[
\sum_{k=0}^j{j\brack k}_q A'_k=\sum_{i=0}^nA_{n-i}\prod_{k=0}^{j-1}(q^i-q^k).
\]
Since $Y$ is an $(n-t)$-code, we have $A_1=\cdots=A_{n-t-1}=0$ and, since $Y$ is a $t$-design, we find by Proposition~\ref{pro:t-designs_dual_dist} that $A'_1=\cdots=A'_t=0$. Moreover we have $A_0=1$ and $A'_0=\abs{Y}$ and therefore
\[
\abs{Y}-\prod_{k=0}^{j-1}(q^n-q^k)=\sum_{i=0}^tA_{n-i}\prod_{k=0}^{j-1}(q^i-q^k)
\]
for each $j\in\{1,2,\dots,t\}$. The identity
\[
\prod_{k=0}^{j-1}(q^i-q^k)={i\brack j}_q(q^j-1)\cdots(q^j-q^{j-1})
\]
gives
\begin{align*}
\sum_{i=0}^tA_{n-i}{i\brack j}_q&=\frac{\abs{Y}-\prod_{k=0}^{j-1}(q^n-q^k)}{\prod_{k=0}^{j-1}(q^j-q^k)}\\
&={n\brack j}_q\bigg(\frac{\abs{Y}}{\prod_{k=0}^{j-1}(q^n-q^k)}-1\bigg)
\end{align*}
for each $j\in\{1,2,\dots,t\}$. Now the inversion formula~\eqref{eqn:inversion_formula} gives the desired result.
\end{proof}
\par
\begin{remark}
\label{rem:Rudvalis_Shinoda}
Consider $Y=\GL(n,q)$ with inner distribution $(A_i)$, so that $A_{n-i}=w_i$ for all $i$. Since  $Y$ is a $1$-code and an $n$-design, Theorem~\ref{thm:inner_dist} gives
\[
A_{n-i}=\sum_{j=i}^{n-1}(-1)^{j-i}q^{j-i\choose 2}{j\brack i}_q{n\brack j}_q\Bigg(\prod_{k=j}^{n-1}(q^n-q^k)-1\Bigg).
\]
Now a lengthy, but straightforward, calculation reveals that $A_{n-i}=w_i$, given in~\eqref{eqn:wi}. Note that the proof of Theorem~\ref{thm:inner_dist} uses only the (easy) forward direction of Proposition~\ref{pro:t-designs_dual_dist} and not the decomposition in Theorem~\ref{thm:decomposition_AlSC_chars}. Hence our proof of Theorem~\ref{thm:inner_dist} and therefore of~\eqref{eqn:wi} is self-contained. 
\end{remark}
\par
Note that the upper bound in Corollary~\ref{cor:bound_codes} is at most
\[
q^{n(n-d+1)}.
\]
We close this section by showing that there exist $d$-codes almost as large as this upper bound. Our construction uses so-called \emph{linear maximum rank distance codes} with minimum distance $d$, which are $\FF_q$-subspaces $Z$ of $\FF_q^{n\times n}$ of dimension $n(n-d+1)$, such that $\rk(x-y)\ge d$ for all distinct $x,y\in Z$. Such objects exist for all integers $d$ satisfying $1\le d\le n$~\cite[Theorem~6.3]{Del1978}.
\begin{proposition}
For each $d$ satisfying $1\le d\le n$, there exists a $d$-code in $\GL(n,q)$ of size at least
\[
\left(1-\frac{1}{q-1}\right)q^{n(n-d+1)}.
\]
For $q=2$ there exists a $d$-code in $\GL(n,q)$ of size at least $q^{n(n-d)}$.
\end{proposition}
\begin{proof}
Consider a linear maximum rank distance code $Z$ in $\FF_q^{n\times n}$ with minimum distance $d$. We show that 
$Z\cap\GL(n,q)$ has the required properties. It is well known~\cite[Theorem~5.6]{Del1978} that the number of matrices in $Z$ of rank $i$ depends only on the parameters $q$, $n$, and $d$. In particular the number of invertible matrices in $Z$ equals
\[
N=\sum_{j=0}^{n-d}(-1)^jC_j,
\]
where
\[
C_j=q^{j\choose 2}{n\brack j}_q(q^{n(n-d+1-j)}-1).
\]
It is readily verified that $C_j/C_{j+1}\ge q^{j-1}$ and therefore $C_0,C_1,\dots$ is nonincreasing. Hence we have
\begin{align*}
N\ge C_0-C_1&=(q^{n(n-d+1)}-1)-\frac{q^n-1}{q-1}(q^{n(n-d)}-1)\\
&\ge\frac{q-2}{q-1}q^{n(n-d+1)}+\frac{q^{n(n-d)}-1}{q-1},
\end{align*}
as required
\end{proof}



\providecommand{\bysame}{\leavevmode\hbox to3em{\hrulefill}\thinspace}
\providecommand{\MR}{\relax\ifhmode\unskip\space\fi MR }
\providecommand{\MRhref}[2]{%
  \href{http://www.ams.org/mathscinet-getitem?mr=#1}{#2}
}
\providecommand{\href}[2]{#2}

\end{document}